\documentclass[a4paper, 12pt]{article}

\usepackage[utf8]{inputenc}
\usepackage[russian, english]{babel}
\usepackage[T1]{fontenc}
\usepackage[a4paper, margin=2.5cm]{geometry}
\usepackage{needspace}

\usepackage{libertine} 
\usepackage{inconsolata} 

\usepackage{amsmath, amsthm, amssymb}
\usepackage{graphicx}
\usepackage{enumerate}
\usepackage{authblk}
\usepackage[textsize=small, textwidth=2cm, color=yellow]{todonotes}
\usepackage{thmtools}
\usepackage{thm-restate}
\usepackage[colorlinks=true, citecolor=red]{hyperref}
\usepackage{float}

\declaretheorem[name=Theorem, numberwithin=section]{theorem}
\declaretheorem[name=Lemma, sibling=theorem]{lemma}
\declaretheorem[name=Proposition, sibling=theorem]{proposition}

\declaretheorem[name=Corollary, sibling=theorem]{corollary}
\declaretheorem[name=Conjecture, sibling=theorem]{conjecture}
\declaretheorem[name=Problem, sibling=theorem]{problem}

\def\cqedsymbol{\ifmmode$\lrcorner$\else{\unskip\nobreak\hfil
\penalty50\hskip1em\null\nobreak\hfil$\lrcorner$
\parfillskip=0pt\finalhyphendemerits=0\endgraf}\fi}

\newcommand{\wFunc}[1]{{g_N\left(#1\right)}}

\let\le\leqslant
\let\ge\geqslant
\let\leq\leqslant
\let\geq\geqslant

\newcommand{\dt}{dt}

\title{Prime and polynomial distances in colourings\\ of the plane\thanks{This work is a part of projects CUTACOMBS (R. McCarty) and BOBR (Mi. Pilipczuk) that have received funding from the European Research Council (ERC) under the European Union’s Horizon 2020 research and innovation programme (grant agreement No. 714704 and No. 948057, respectively).
Rose McCarty is also supported by the National Science Foundation under Grant No. DMS-2202961.}}

\author[1]{James Davies}
\author[2]{Rose McCarty}
\author[3]{Micha{\l} Pilipczuk}

\affil[1]{University of Cambridge, United Kingdom.}
\affil[2]{Georgia Institute of Technology, United States.\footnote{Previously at Princeton University, United States.}}
\affil[3]{University of Warsaw, Poland.}
\date{}

\begin{document}

\maketitle

\begin{abstract}
    We give two extensions of the recent theorem of the first author that the odd distance graph has unbounded chromatic number. The first is that for any non-constant polynomial $f$ with integer coefficients and positive leading coefficient, every finite colouring of the plane contains a monochromatic pair of distinct points whose distance is equal to $f(n)$ for some integer $n$. The second is that for every finite colouring of the plane, there is a monochromatic pair of points whose distance is a prime number.
\end{abstract}

\section{Introduction}

Furstenberg~\cite{furstenberg1977ergodic} and S{\'a}rk{\"o}zy~\cite{sarkozy1978difference} independently proved that given a non-zero polynomial $f$ with integer coefficients and $f(0)=0$, every finite colouring of the integers contains a pair of distinct monochromatic integers which are of the form $a$ and $a + f(n)$ for some $n\in \mathbb{Z}$.
The condition that $f(0)=0$ is not a necessary one; Kamae and Mend{\`e}s France~\cite{kamae1978van} proved that this condition can be replaced by the more general condition that for every non-zero $k\in \mathbb{N}$, we have $f(\mathbb{Z}) \cap k\mathbb{Z} \not= \emptyset$.
This condition is necessary, since otherwise by colouring the integers according to their residues modulo $k$, we could avoid such monochromatic pairs.

Note in particular that there is a $2$-colouring of the integers that avoids monochromatic pairs of the form $a$ and $a+(2n+1)$, in other words, pairs whose difference is odd.
Rosenfeld~\cite{erdos1994twenty} asked whether the situation is different in the plane. More precisely, he asked how many colours are required to colour $\mathbb{R}^2$ so that there are no two monochromatic points whose distance is an odd integer. As with the Hadwiger-Nelson problem (see~\cite{SoiferHistory2009} for a history of this well-known problem), small unit distance graphs such as the Moser spindle give a lower bound of~$4$. Ardal, Ma{\v{n}}uch, Rosenfeld, Shelah, and Stacho~\cite{ardal2009odd} improved this lower bound by constructing a $5$-chromatic odd distance graph. Then, in a breakthrough result, de Grey~\cite{de2018chromatic} and Exoo and Ismailescu~\cite{exoo2020chromatic} independently constructed $5$-chromatic unit distance graphs. Recently, Parts~\cite{parts20226} constructed an odd distance graph requiring at least $6$ colours.

Rosenfeld's odd distance problem was very recently solved by the first author~\cite{davies2022odd}, who proved that infinitely many colours are required. This shows a fundamental difference between the chromatic number of the $f(n)=2n+1$ distance graph on $\mathbb{R}^2$ and on $\mathbb{Z}$ (or equivalently on $\mathbb{R}$, since one can extend a colouring on $\mathbb{Z}$ to $\mathbb{R}$ by assigning each interval $(n,n+1)$ the same colour as $n$). In fact, the first author~\cite{davies2022odd} proved this result about the plane in a more general setting, for any linear polynomial $f(n)=an+b$ over $\mathbb{Z}$. Extending the methods of~\cite{davies2022odd}, we further generalize this result to essentially any polynomial over $\mathbb{Z}$.

\begin{theorem}\label{poly dist}
    Let $f(x)=a_rx^r + \cdots + a_1x+ a_0$ be a non-constant polynomial with integer coefficients and with $a_r \ge 1$.
	Then every finite colouring of the plane contains a monochromatic pair of distinct points whose distance from each other is equal to $f(n)$ for some integer $n$.
\end{theorem}

The condition that $a_n\ge 1$ is simply to exclude the degenerate case that $f(x)$ is a polynomial with only a bounded range of positive values (in which case, we can obtain a finite colouring of the plane avoiding such monochromatic pairs by taking a product of colourings where each avoids monochromatic pairs of one given distance).

Eggleton, Erd{\H{o}}s, and Skilton~\cite{eggleton1990colouring} also studied the prime distance graph on $\mathbb{Z}$ (or equivalently on $\mathbb{R}$, for the same reason as before). Here, the chromatic number is at most~$4$ since $4\mathbb{Z}$ contains no primes. (In fact the chromatic number is exactly~$4$, as observed in~\cite{eggleton1990colouring}.)
The situation again is different in the plane; we show that there is no finite colouring of $\mathbb{R}^2$ avoiding monochromatic pairs of points whose distance is a prime number.

\begin{theorem}\label{prime dist}
    Every finite colouring of the plane contains a monochromatic pair of points whose distance from each other is a prime number.
\end{theorem}

We remark that if we require the colour classes to be measurable sets, then more is known.
F{\"u}rstenberg, Katznelson, and Weiss~\cite{furstenberg1990ergodic} gave an ergodic-theoretic proof that for every finite measurable colouring of the plane, there exists $d_0 > 0$ such that for any real number $d \ge d_0$, there is a monochromatic pair of points at distance $d$ from each other. For other proofs of this result, see~\cite{bourgain1986szemeredi,bukh2008measurable,de2010fourier,falconer1986plane}.

Our method to prove Theorem~\ref{poly dist} and Theorem~\ref{prime dist} extends the spectral method of~\cite{davies2022odd} on Cayley graphs of $\mathbb{Z}^d$ by combining it with Vinogradov's~\cite{nathanson1996additive,vinogradov1937representation,vinogradow1937some,vinogradov1984method} estimates for exponential sums.
In Section~\ref{sec:lattices}, we introduce some necessary preliminaries including the spectral ratio bound theorem of~\cite{davies2022odd}. We also show that certain Cayley graphs of $\mathbb{Z}^2$ embed as distance graphs in $\mathbb{R}^2$, and give a couple of additional preliminary lemmas.
Then, we prove Theorem~\ref{poly dist} and Theorem~\ref{prime dist} in Section~\ref{sec:poly} and Section~\ref{sec:prime}, respectively.
Lastly, in Section~\ref{sec:con}, we discuss some open problems, and we give an obstruction to generalizing these theorems to triples of points (see Proposition~\ref{line}).

\section{Lattices and other preliminaries}\label{sec:lattices}

We say that a set $C\subseteq \mathbb{Z}^2$ is \emph{centrally symmetric} if $C=-C$. Similarly, a function $w: C \to \mathbb{R}$ is \emph{centrally symmetric} if $w(x)=w(-x)$ for all $x\in C$. For a centrally symmetric set $C\subseteq \mathbb{Z}^2 \backslash \{0\}$, we let $G(\mathbb{Z}^2, C)$ be the Cayley graph of $\mathbb{Z}^2$ with generating set $C$.
This is the graph with vertex set $\mathbb{Z}^2$, where two vertices $u,v\in \mathbb{Z}^2$ are adjacent if there exists a $c\in C$ with $u+c=v$.
Given a set $I\subseteq \mathbb{Z}^2$, its \emph{upper density} is
\[
\delta(I) = \limsup_{R \to \infty} \frac{|I \cap [-R,R]^2  |}{(2R+1)^2}.
\]

An \emph{independent set} of a graph $G$ is a set of pairwise non-adjacent vertices.
For a Cayley graph $G=G(\mathbb{Z}^2 , C)$, its \emph{independence density} is
\[
\overline{\alpha}(G) = \sup \{ {\delta(I)} : I \text{ is an independent set of } G \}.
\]
In particular, notice that $\chi(G) \ge 1/ \overline{\alpha}(G)$, where $\chi(G)$ denotes the chromatic number of $G$, and we consider the right hand side to be $\infty$ if $\overline{\alpha}(G)=0$. To see this inequality, note that if $V(G)$ can be partitioned into $k$ independent sets $I_1, \ldots, I_k$, then $k\overline{\alpha}(G) \geq \delta(I_1)+\dots+\delta(I_k) \geq \delta(V(G))=1$. Thus, we can focus on showing that the independence density is small.

Next we state a slightly simplified version of the ratio bound given in \cite[Theorem~3]{davies2022odd}, which was used to show that the odd distance graph has unbounded chromatic number. This is essentially an analogue of the Lov{\'a}sz theta bound~\cite{lovasz1979shannon} (or a weighted version of Hoffman's ratio bound; see~\cite{hoffman}) for Cayley graphs of $\mathbb{Z}^2$. It gives an upper bound for $\overline{\alpha}$. Informally, this upper bound also allows us some freedom to choose a ``good weight function'' $w$ of the generating set $C$.

To state this theorem, it is convenient to give some additional notation. First, for $\alpha \in \mathbb{R}$, we let $e(\alpha) = e^{2 \pi i \alpha}$. Notice that $e(\alpha)+e(-\alpha)=2\cos(2\pi \alpha)$ is a real number between $-2$ and $2$; this is one way we will use central symmetry. We write $u \cdot x$ for the dot product. Now, given a finite set $C\subseteq \mathbb{Z}^2 \backslash \{0\}$ and $w:C \to \mathbb{R}_{\ge 0}$ which are centrally symmetric, we write $\widehat{w}$ for the function where, for each $u \in \mathbb{R}^2$, we set
\[
	\widehat{w}(u) = \sum_{x\in C} w(x) e( u \cdot x).
\]
\noindent Notice that since $C$ and $w$ are centrally symmetric and $w:C \to \mathbb{R}_{\ge 0}$, the function $\widehat{w}$ is real-valued and attains its supremum at $\widehat{w}(0) =\sum_{x\in C} w(x)$.

Now we are ready to state the key theorem.

\begin{theorem}[Davies {\cite[Theorem~3]{davies2022odd}}]\label{ratio bound}
    For any finite set $C\subseteq \mathbb{Z}^2 \backslash \{0\}$ and any function $w:C \to \mathbb{R}_{\ge 0}$ so that $w$ and $C$ are centrally symmetric and $\sum_{x\in C} w(x)$ is positive, we have
	\[
	\overline{\alpha}(G(\mathbb{Z}^2, C)) \le \frac{-\inf_{u\in \mathbb{R}^2}  \widehat{w}(u) }{  \sup_{u\in \mathbb{R}^2}  \widehat{w}(u)   -\inf_{u\in \mathbb{R}^2}  \widehat{w}(u)}.
	\]
\end{theorem}

\noindent With the appropriate conditions, this theorem can be formulated when $C$ is infinite as well. This theorem implies that, informally, the chromatic number of $G(\mathbb{Z}^2, C)$ is large when the quantity $\sum_{x\in C} w(x)$, which is the supremum, is many times larger than $-\inf_{u\in \mathbb{R}^2}  \widehat{w}(u)$.
Note that $\inf_{u\in \mathbb{R}^2}  \widehat{w}(u)$ is always negative, since $\int_{[0,1]^2} \widehat{w}(u) \,du = 0$.

Now we show that certain Cayley graphs over $\mathbb{Z}^2$ are subgraphs of certain distance graphs in the plane. For positive integers $q$ and $k$ with $k \geq 2$, we let
\[
X_{q,k} = \prod_{j=1}^{k-1} (q^{k+j} +  q^{k-j}   +1  ),
\]
so that $X_{q,k}$ is a positive integer. We also let $D_{q,k}$ be the set which contains, for each $j \in \{1,2\ldots, k-1\}$, the following element of $\mathbb{Z}^2$:
\[
\frac{X_{q,k}}{q^{k+j} +  q^{k-j}   +1} \left(
      q^{k+j}-q^{k-j} ,
      q^{j} + 2q^{k}
    \right).
\]
\noindent Given a set $A \subseteq \mathbb{N}$, we write $\pm AD_{q,k}$ for the set of all elements of $\mathbb{Z}^2$ which are of the form $\pm ad$ for some $a \in A$ and $d \in D_{q,k}$. 

We show that for any $A$, the Cayley graph generated by $\pm AD_{q,k}$ is contained in the distance graph of the plane with distance set $A$.

\begin{lemma}\label{embed}
For any positive integers $q$ and $k$ with $k \geq 2$ and for any $A\subseteq \mathbb{N}$, the graph $G(\mathbb{Z}^2, \pm AD_{q,k})$ is a subgraph of the graph on vertex set $\mathbb{R}^2$ where two vertices are adjacent if their distance is contained in $A$.
\end{lemma}

\begin{proof}
    Let us write $G=G(\mathbb{Z}^2, \pm AD_{q,k})$ for convenience. We will use the following two basis vectors of $\mathbb{R}^2$ to define an injection from $V(G)$ to $\mathbb{R}^2$. Let
\[
e_1 = \frac{1}{X_{q,k}} (1,0),
\hspace{.3cm}\textrm{ and }\hspace{.3cm}
e_2 = \frac{1}{X_{q,k}} \left(-\frac{1}{2q^{k}} , \sqrt{1 - \frac{1}{4q^{2k}}} \right).
\]
Then, let $h:\mathbb{Z}^2 \to \mathbb{R}^2$ be defined by setting $h(x,y) = xe_1 + ye_2$ for each $(x,y) \in \mathbb{Z}^2$. Clearly $h$ is injective, so it is enough to show that if $(x,y)$ is adjacent to $(x',y')$ in $G$, then the distance between $h(x,y)$ and $h(x',y')$ is contained in $A$.

	So, suppose that $(x,y)$ is adjacent to $(x',y')$ in $G$. Then $(x-x',y-y')\in \pm AD_{q,k}$, and thus there exists a pair of integers $a\in A$ and $j \in \{1, \ldots, k-1\}$ such that
	\[
	(x-x',y-y') = 
	\frac{\pm a X_{q,k}}{q^{k+j} + q^{k-j} + 1} \left(
	q^{k+j}-q^{k-j} ,
	q^{j} + 2q^{k}
	\right).
	\]
 
        \noindent Then, 
	\begin{align*} 
	&
	\left(
	\frac{q^{k+j} + q^{k-j} + 1}{a}
	\right)^2
	\left\| h(x,y) - h(x',y')\right\|^2
    \\
	& \qquad\qquad =
	\left\|
	\left(\frac{q^{k+j} + q^{k-j} + 1}{a}\right)\left((x-x')e_1+(y-y')e_2 \right)
	\right\|^2
	\\
	& \qquad\qquad =
	\left\|
	(q^{k+j}-q^{k-j})\left(1,0\right)  +
	(q^{j} + 2q^{k}) \left(-\frac{1}{2q^{k}} , \sqrt{1 - \frac{1}{4q^{2k}}} \right)
	\right\|^2
	\\
        & \qquad\qquad = \left[(q^{k+j}-q^{k-j}) - \frac{1}{2q^{k}}  (q^{j} + 2q^{k})  \right]^2 + \left[ \sqrt{1-\frac{1}{4q^{2k}}} (q^{j} + 2q^{k}) \right]^2  
        \\
	& \qquad\qquad = (q^{k+j}-q^{k-j})^2 +  (q^{j} + 2q^{k})^2 - \frac{1}{q^{k}} (q^{k+j}-q^{k-j}) (q^{j} + 2q^{k})  
        \\
	& \qquad\qquad = (q^{2k+2j}+ q^{2k-2j} -2q^{2k})  
	+ (q^{2j} +4q^{2k} +4q^{k+j})  
	+ (-q^{2j}  + 2q^{k-j}  -2 q^{k+j}  + 1  )    
        \\
	& \qquad\qquad = q^{2k+2j} + q^{2k-2j}  + 2q^{2k}   +2q^{k+j} + 2q^{k-j}    + 1 
        \\
	& \qquad\qquad = (q^{k+j} +  q^{k-j}   + 1  )^2.
     \end{align*}
    Hence $\left\| h(x,y) - h(x',y')\right\| = a$, as desired.
\end{proof}

We now show that, as long as $q \geq 2$, no two points in $D_{q,k}$ are scalar multiples of each other. This lemma will be useful later on; for instance it allows us to conclude that $|AD_{q,k}|=|A||D_{q,k}|$ for any $A \subseteq \mathbb{Z}$.

\begin{lemma}\label{not scalar multiples}
For any positive integers $k$ and $q$ with $q \geq 2$, no two points in $D_{q,k}$ are scalar multiples of each other.
\end{lemma}
\begin{proof}
    Going for a contradiction, suppose otherwise. Then there exist distinct $P, P' \in D_{q,k}$ so that the ratio of the two coordinates of $P$ is the same as for $P'$. Using the definition of $D_{q,k}$, this means that there exist distinct $i, j \in \{1,2,\ldots, k-1\}$ so that \[
    \frac{q^{k+j}-q^{k-j}}{q^{j} + 2q^{k}} =\frac{q^{k+i}-q^{k-i}}{q^{i} + 2q^{k}} .
    \] 
    \noindent From the above, we have that \[
    \left(q^{k+j}-q^{k-j}\right) \left(2q^k+q^i \right)=\left(q^{k+i}-q^{k-i}\right) \left(2q^k+q^j \right).
    \] Without loss of generality we may assume that $i<j$. Then $q^{k-j+i}$ is the highest power of $q$ which divides the left-hand side, but $q^{k-i+j}$, which is larger, divides the right-hand side. This contradiction shows that indeed no two points in $D_{q,k}$ are scalar multiples of each other.
\end{proof}

Now we prove a lemma that will be used to help examine certain exponential sums in the proofs of both Theorem~\ref{poly dist} and Theorem~\ref{prime dist}. This lemma will be used in combination with Lemma~\ref{embed}, where we will set $q = Q!$.

\begin{lemma}\label{sandwich}
For any integers $Q,k \geq 2$, there exists $\epsilon_{Q,k}>0$ such that for every $u\in \mathbb{R}^2$, there are at most four distinct elements $d \in \pm D_{{Q!},k}$ so that $\left|u \cdot d - \frac{a}{b} \right| < \epsilon_{Q,k}$ for some non-zero coprime integers $a$ and $b$ with $2\le b\le Q$. Moreover, the set of all such $d\in \pm D_{{Q!},k}$ is centrally symmetric.
\end{lemma}

\begin{proof}
    For convenience, we set $q=Q!$ so that $D_{q,k}=D_{{Q!},k}$. Recall that $D_{q,k}$ consists of $k-1$ points in $\mathbb{Z}^2$, where for each $j \in \{1,2,\ldots, k-1\}$, we have the point\[
    P_j=\frac{X_{q,k}}{q^{k+j} +  q^{k-j}   +1} \left(
      q^{k+j}-q^{k-j} ,
      q^{j} + 2q^{k}
    \right).
    \] 

    \noindent Now we define $\epsilon_{Q,k}$. Since no two points in $D_{q,k}$ are scalar multiples of each other by Lemma~\ref{not scalar multiples}, it follows that for any triple of distinct $j_1, j_2, j_3 \in \{1,2, \ldots, k-1\}$, there exist non-zero integers $c_1,c_2,c_3$ such that $c_1P_{j_1} + c_2P_{j_2} + c_3P_{j_3} = 0$. For each such triple, let $m_{j_1, j_2, j_3}$ denote the minimum of $|c_1| + |c_2| + |c_3|$ over all such $c_1,c_2,c_3$. Then, finally, define $\epsilon_{Q,k}$ to be the minimum, over all such triples $j_1, j_2, j_3$, of the number $(Qm_{j_1, j_2, j_3})^{-1}$. So we have $\epsilon_{Q,k}>0$.
    
	Now fix $u \in \mathbb{R}^2$.
	Note that if $\left|u \cdot d - \frac{a}{b} \right| < \epsilon_{Q,k}$, then $\left|u \cdot (-d) - \frac{(-a)}{b} \right| < \epsilon_{Q,k}$.
	So it is enough to show that there are at most two distinct elements $d\in D_{q,k}$ so that $u \cdot d$ is in the interval $\left(\frac{a}{b} - \epsilon_{Q,k} , \frac{a}{b} + \epsilon_{Q,k}\right)$ for some non-zero coprime integers $a$ and $b$ with $2\le b \le Q$.
	Going for a contradiction, suppose otherwise. Then there exist distinct $j_1, j_2, j_3 \in \{1,2,\ldots, k-1\}$ such that for each $i \in \{1,2,3\}$, there are non-zero coprime integers $a_i$ and $b_i$ with $2\le b_i \le Q$ and
	\[
	\left| u \cdot P_{j_i} -  \frac{a_i}{b_i} \right| 
	< \epsilon_{Q,k}.
	\] 
    \noindent We may assume without loss of generality that $j_1<j_2<j_3$.

    Let $c_1, c_2, c_3$ be non-negative integers which achieve the minimum $m_{j_1, j_2, j_3}$, that is, so that $c_1P_{j_1} + c_2P_{j_2} + c_3P_{j_3} = 0$ and $m_{j_1, j_2, j_3}=|c_1| + |c_2| + |c_3|$. Since the first sum is $0$, by looking at the first coordinates of $c_1P_{j_1}$, $c_2P_{j_2}$, and $c_3P_{j_3}$, we have
    \begin{align*} 
        0 &=  X_{q,k}\left[
        c_1 \left( \frac{q^{k+j_1}-q^{k-j_1}}{q^{k+j_1}+q^{k-j_1}+1}\right)
    	+
    	c_2 \left( \frac{q^{k+j_2}-q^{k-j_2}}{q^{k+j_2}+q^{k-j_2}+1} \right)
    	+
    	c_3 \left( \frac{q^{k+j_3}-q^{k-j_3}}{q^{k+j_3}+q^{k-j_3}+1}\right)\right].
    \end{align*}
    \noindent Recall that $X_{q,k} = \Pi_{j=1}^{k-1}(q^{k+j}+q^{k-j}+1)$. Observe that $q$ does not divide any of the terms in this product, but each of the three denominators above does. Also note that since $j_1<j_2$, we have that $q^{k-{j_2}}$ divides both $q^{k+j_1}-q^{k-j_1}$ and $q^{k+j_2}-q^{k-j_2}$. Considering the equation above mod$(q^{k-{j_2}})$, it follows that $q^{k-{j_2}}$ divides $c_3(q^{k+j_3}-q^{k-j_3})$. Thus $q$ divides $c_3$, since $j_2<j_3$.

    Similarly, but now using the second coordinates of the points, we find that there exist integers $x_1, x_2, x_3$, none of which are divisible by $q$, so that
    \begin{align*}
        0 =
        x_1c_1 \left(q^{j_1} + 2q^{k} \right)
    	+
    	x_2c_2 \left(q^{j_2} + 2q^{k} \right)
    	+
    	x_3c_3 \left(q^{j_3} + 2q^{k} \right).
    \end{align*}
    \noindent Since $j_1 < j_2 < j_3$, we find that $q^{j_2}$ divides $q^{j_2} + 2q^{k}$ and $q^{j_3} + 2q^{k}$ but not $q^{j_1} + 2q^{k}$. Thus $q$ divides $c_1$ as well. By the minimality of $|c_1|+|c_2|+|c_3|$, we have that $\gcd(c_1,c_2,c_3)=1$. So, it now follows that $c_{2}$ is coprime to $q$.
	
	Since $q=Q!$ divides both $c_1$ and $c_3$, and since $b_1,b_3\le Q$, we have that $c_1\frac{a_1}{b_1}   +  c_3\frac{a_3}{b_3}$ is an integer. Since $b_2$ divides $Q!$, and since $c_2$ is coprime to $q=Q!$, it follows that $c_2$ is coprime to $b_2$. So, since $a_2$ and $b_2$ are coprime, $c_2a_2$ is coprime to $b_2$. It follows that $a=b_2\left(c_1\frac{a_1}{b_1} + c_2\frac{a_2}{b_2}  +  c_3\frac{a_3}{b_3}\right)$ is an integer that is coprime to $b_2$. Since $b_2 \geq 2$, this in particular means that $a \neq 0$.
	Thus,
	\[
	\left|u \cdot \left( c_1P_{j_1} + c_2P_{j_2} + c_3P_{j_3} \right) - \left( c_1\frac{a_1}{b_1} + c_2\frac{a_2}{b_2}  +  c_3\frac{a_3}{b_3} \right) \right|
	=
	\left|c_1\frac{a_1}{b_1} + c_2\frac{a_2}{b_2}  +  c_3\frac{a_3}{b_3} \right|
	=
	\left|
	\frac{a}{b_2}
	\right|
	\ge \frac{1}{Q}.
	\]
	However, this contradicts the fact that
	\begin{align*}
	    &   \left|u \cdot \left( c_1P_{j_1} + c_2P_{j_2} + c_3P_{j_3} \right) - \left( c_1\frac{a_1}{b_1} + c_2\frac{a_2}{b_2}  +  c_3\frac{a_3}{b_3} \right) \right|
	    \\
	    & \qquad \qquad \qquad \le
	    c_1 \left|
	u \cdot P_{j_1} - \frac{a_1}{b_1}
	\right|
	+
	c_2 \left|
	u \cdot P_{j_2} - \frac{a_2}{b_2}
	\right|
	+
	c_3 \left|
	u \cdot P_{j_3} - \frac{a_3}{b_3}
	\right|
	\\
	& \qquad \qquad \qquad <
	(|c_1|+|c_2|+|c_3|) \epsilon_{Q,k}
	\\
	& \qquad \qquad \qquad \le 
	\frac{1}{Q}.
	\end{align*}
	Hence, there are at most two such elements $d\in D_{Q!,k}$, as desired.
\end{proof}

\section{Polynomial distances}\label{sec:poly}

In this section we shall prove Theorem~\ref{poly dist}. 

We will use the following slightly simplified version of Vinogradov's estimate for exponential sums with polynomial exponents from~\cite{vinogradov1984method}. For convenience, given an integer $r\geq 3$, we set $\rho_r = (8r^2(\ln r +   1.5 \ln \ln r + 4.2))^{-1}$. Thus $\rho_r$ is in the interval $(0,1)$.

\begin{theorem}[Vinogradov {\cite[page 14]{vinogradov1984method}}]\label{poly estimate}
    For any $\epsilon >0$ and polynomial $f$ of degree $r\ge 3$ with integer coefficients, there exists a real number $c(f, \epsilon)>0$ satisfying the following. Let $\alpha \in \mathbb{R}$ and $N,M\in \mathbb{N}$. If there exist coprime integers $a$ and $b$ with $1 \le b\le N^{\frac{1}{r}}$ and $\left| \alpha - \frac{a}{b} \right| \le N^{-r+\frac{1}{r}}$, then
\begin{align}
\label{polyEst:First}
\left| \sum_{j=1}^N e(\alpha f(j+M)) \right| \le c(f, \epsilon) N b^{-\frac{1}{r}+ \epsilon},
\end{align}
 and otherwise
\begin{align}
\label{polyEst:Second}
\left |\sum_{j=1}^N e(\alpha f(j+M)) \right| \le c(f, \epsilon) N^{1- {\rho_r}}.
\end{align}
\end{theorem}
\noindent We note that the constant $c(f, \epsilon)$ actually only depends on the leading coefficient of $f$, however the statement above is all we will need.

Now, note that it is enough to prove Theorem~\ref{poly dist} for polynomials $f$ whose degree is at least~$3$, since otherwise we can consider the polynomial $f(x^3)$. Similarly, it is enough to prove this theorem for polynomials $f$ which are positive and strictly increasing on $[0,\infty)$, because otherwise we can consider the polynomial $f(x+M)$ for a sufficiently large integer $M$. So, for the rest of this section, we fix a polynomial $f=a_rx^r+\ldots+a_1x+a_0$ with integer coefficients so that $r \ge 3$, $a_r \ge 1$, and $f$ is positive and strictly increasing on $[0,\infty)$. 

Our strategy is as follows. Let us fix integers $Q,k\ge 2$. For each positive integer $N$, we set $A_N=\{f(1), f(2), \ldots, f(N)\}$. In the next paragraph, we will define a centrally symmetric weight function $w_N$ on the set $\pm A_ND_{Q!, k}$ (which is also centrally symmetric). Then we will apply Theorem~\ref{ratio bound} to the Cayley graph $G(\mathbb{Z}^2, \pm A_ND_{Q!, k})$ with weights $w_N$. In this manner, we will be able to bound the independence density of $G(\mathbb{Z}^2, \pm A_ND_{Q!, k})$. Then, by considering larger and larger $N$, we will be able to bound the independence density of the Cayley graph $G(\mathbb{Z}^2, \pm AD_{Q!, k})$, where $A = f(\mathbb{N})$. This will be enough since, by Lemma~\ref{embed}, the Cayley graph $G(\mathbb{Z}^2, \pm AD_{Q!, k})$ is a subgraph of the desired distance graph of the plane.

If $Q$ and $k$ are clear from context, then for each $j \in \{1,2,\ldots, N\}$ and $d \in \pm D_{Q!, k}$, we write \begin{align*}
w_N(f(j)d) = \frac{f'(j) \left(  1 - \frac{f(j)}{f(N)}  \right)}{\int_{0}^{N} f'(t) \left(  1 - \frac{f(t)}{f(N)}  \right) \,dt},
\end{align*}
\noindent where $f'$ denotes the derivative of $f$. Note that $w_N$ is well-defined since $f$ is positive and increasing on $[0,\infty)$, and since no two elements in $D_{Q!, k}$ are scalar multiples by Lemma~\ref{not scalar multiples}. Moreover, $w_N$ is centrally symmetric and the value of $w_N(f(j)d)$ only depends on~$j$. For this reason and for other upcoming calculations, it is convenient to define a similar function on $\mathbb{R}$. Thus, for any $j \in \mathbb{R}$, we write \begin{align*}\wFunc{j} = \frac{f'(j) \left(  1 - \frac{f(j)}{f(N)}  \right)}{\int_{0}^{N} f'(t) \left(  1 - \frac{f(t)}{f(N)}  \right) \,\dt}.\end{align*}

Recall that, in order to apply Theorem~\ref{ratio bound}, we need to bound the infimum and the supremum of the function $\widehat{w}_N$ whose value, given $u\in \mathbb{R}^2$, is the sum over all $x \in \pm A_ND_{Q!, k}$ of $w_N(x)e(u\cdot x)$. Re-writing this sum, we have\begin{align*}
 \widehat{w}_N(u) = \sum_{d \in \pm D_{Q!, k}}\sum_{j=1}^N g_N(j)e(u\cdot f(j)d).
\end{align*}
\noindent We frequently bound $\widehat{w}_N$ by fixing $d$ and considering the inner sum, which is of the form $\sum_{j=1}^N g_N(j)e(\alpha f(j))$ for some real number $\alpha$. This is the connection to Vinogradov's Theorem. Another technique we often use is to show that a sum is close to a corresponding integral. For this method we use the following bound on the derivative of $g_N$. 

\begin{lemma}
\label{lem:polyDer}
For any fixed $\epsilon>0$, \begin{align*}
\lim_{N \rightarrow \infty} N^{2-\epsilon} \sup_{x \in [0,N]} |g_N'(x)| = 0.
\end{align*}
\end{lemma}
\begin{proof} First, note that the derivative evaluates to\begin{align*}
g_N'(x)= \frac{f''(x)\left(1-\frac{f(x)}{f(N)}\right)+f'(x)\left(-\frac{f'(x)}{f(N)}\right)}{\int_{0}^{N} f'(t) \left(  1 - \frac{f(t)}{f(N)}  \right) \,\dt} = \frac{f''(x)-\frac{f''(x)f(x)}{f(N)}-\frac{f'(x)f'(x)}{f(N)}}{\int_{0}^{N} f'(t) \left(  1 - \frac{f(t)}{f(N)}  \right) \,\dt}.
\end{align*}For the numerator, the following inequality holds when $N$ is sufficiently large since $f$ is a polynomial with positive leading coefficient: \begin{align*}
\sup_{x \in [0,N]}\left| f''(x)-\frac{f''(x)f(x)}{f(N)}-\frac{f'(x)f'(x)}{f(N)}\right| \leq f''(N)+f''(N)+f'(N)^2/f(N) = \mathcal{O}(N^{r-2}).
\end{align*}Moreover, the integral in the denominator of $g_N'(x)$ is of order $\Omega(N^r)$. 
Putting these things together, we have that $\sup_{x \in [0,N]} |g_N'(x)|= \mathcal{O}\left(1/N^2 \right)$, and the lemma follows.
\end{proof}

Now we use Lemma~\ref{lem:polyDer} to evaluate the following useful sum.

\begin{lemma}
\label{lem:avgPoly} 
We have\begin{align*}\lim_{N \rightarrow \infty} \sum_{j=1}^N g_N(j)=1.
\end{align*}
\end{lemma}
\begin{proof}
    Observe that for every $N \in \mathbb{N}$, we have $\int_{0}^{N} g_N(t)\,\dt=1$. Moreover,\begin{align*}
        \lim_{N \rightarrow \infty} \left|\sum_{j=1}^N g_N(j) -\int_{0}^{N} g_N(t)\,\dt\right|\le \lim_{N \rightarrow \infty}\sum_{j=1}^N \sup_{x\in[j-1,j]}\left|g_N'(x) \right| \le \lim_{N\rightarrow \infty}N\sup_{x\in [0,N]}|g_N'(x)|=0
    \end{align*}by Lemma~\ref{lem:polyDer}. It follows that $\lim_{N \rightarrow \infty} \sum_{j=1}^N g_N(j)=1$, as desired.
\end{proof}

We are ready to compute the supremum we are interested in.

\begin{lemma}
\label{lem:supPoly}
For any fixed integers $Q,k \geq 2$, \begin{align*}\lim_{N \rightarrow \infty} \sup_{u \in \mathbb{R}^2} \widehat{w}_N(u) = 2k-2.
\end{align*}
\end{lemma}
\begin{proof}
    Recall from Section~\ref{sec:lattices} that the supremum is attained at $\widehat{w}_N(0)=\sum_{d \in \pm D_{Q!, k}}\sum_{j=1}^N g_N(j)$. Thus we are done by Lemma~\ref{lem:avgPoly} and the fact that $|\pm D_{Q!, k}|=2|D_{Q!,k}|=2(k-1)$.
\end{proof}

To evaluate the infimum, we will need different techniques depending on whether, informally, $u \cdot d$ is ``near'' an integer, a rational with a small denominator, or neither. So examining these exponential sums splits into three cases given by so called ``integral'', ``major'', and ``minor'' arcs. We note that estimating exponential sums by considering such arcs is common in areas such as analytic number theory; see \cite{nathanson1996additive}, for instance. 

So, let $N$ be a positive integer, and recall that $r$ is the degree of our polynomial $f$. For every $a\in \mathbb{Z}$, we define the \emph{integral arc} $\mathcal{I}_{N}(a) = [a-N^{-r+\frac{1}{r}}, a +N^{-r+\frac{1}{r}}]$, where this notation denotes a closed interval of $\mathbb{R}$.
We write $\mathcal{I}_{N} = \bigcup_{a\in \mathbb{Z}} \mathcal{I}_{N}(a)$ for the union of all integral arcs.
Now, let $Q\ge 2$ be an integer.
For every pair of non-zero coprime integers $a$ and $b$ with $2\le b \le Q$, we define the \emph{major arc}
\[
\mathfrak{M}_{N,Q}(a,b) = \left[\frac{a}{b}-N^{-r+\frac{1}{r}}, \frac{a}{b} + N^{-r + \frac{1}{r}} \right].
\]
Like before, we write $\mathfrak{M}_{N,Q}$ for the union of the above major arcs. We call any interval of $\mathbb{R}$ which is disjoint from $\mathcal{I}_{N}$ and $\mathfrak{M}_{N,Q}$ a \emph{minor arc}, and we write $\mathfrak{m}_{N,Q}= \mathbb{R} \backslash (\mathcal{I}_{N} \cup \mathfrak{M}_{N,Q})$.

First we use Lemma~\ref{sandwich} to handle the exponential sums within the major arcs.

\begin{lemma}
\label{lem:polyMajor} For any fixed integers $Q, k \geq 2$,\begin{align*}
    \liminf_{N \rightarrow \infty} \inf_{u \in \mathbb{R}^2} \sum_{\substack{d \in \pm D_{Q!, k}\\u \cdot d \in \mathfrak{M}_{N,Q}}}\sum_{j=1}^N g_N(j)e(u \cdot f(j)d) \geq -4.
\end{align*}
\end{lemma}
\begin{proof}
    By Lemma~\ref{sandwich}, there exists $\epsilon_{Q,k}>0$ such that for every $u\in \mathbb{R}^2$, there are at most four elements $d\in \pm D_{Q!,k}$ such that $\left|u \cdot d - \frac{a}{b} \right| < \epsilon_{Q,k}$ for some non-zero coprime integers $a$ and $b$ with $2\le b\le Q$. Thus, for any sufficiently large integer $N$, there are at most four elements $d\in \pm D_{Q!,k}$ such that $u \cdot d \in \mathfrak{M}_{N,Q}$. Moreover, again by Lemma~\ref{sandwich}, the set of all such $d$ is centrally symmetric. The lemma then follows from the fact that $e(\alpha)+e(-\alpha)=2\cos(2\pi\alpha)\geq -2$ for any real number $\alpha$, and the fact that $\lim_{N \rightarrow \infty}\sum_{j=1}^N g_N(j) =1$ by Lemma~\ref{lem:avgPoly}.
\end{proof}

Next we handle the exponential sums within the integral arcs.
\begin{lemma}\label{poly zero}
    We have that
    \[
    \liminf_{N \to \infty} \inf_{\alpha\in \mathcal{I}_{N}} \sum_{j=1}^N \wFunc{j} \left[e(\alpha f(j)) + e(-\alpha f(j))\right] \ge 0.
    \]
\end{lemma}
\begin{proof}
It is convenient to write $e(\alpha f(j)) + e(-\alpha f(j))=2\cos(2\pi\alpha f(j))$. Note that this value is the same for $\alpha$ and $\alpha+1$, that is, that $\cos(2\pi\alpha f(j))=\cos(2\pi(\alpha+1) f(j))$ for any $\alpha \in \mathbb{R}$. So it makes no difference to take the infimum over all $\alpha \in \mathcal{I}_{N}(0)$ instead of over all $\alpha \in \mathcal{I}_{N}$. Consequently, it is enough to prove that $\liminf_{N \to \infty} \inf_{\alpha\in {\mathcal{I}_{N}(0)}} \sum_{j=1}^N \wFunc{j} \cos(2\pi\alpha f(j)) \geq 0$. 

We now prove equation~(\ref{eqnPoly}) below. After that, we will evaluate its right-hand side.
\begin{align}
\label{eqnPoly}
    \liminf_{N \to \infty} \inf_{\alpha\in {\mathcal{I}_{N}(0)}} \sum_{j=1}^N \wFunc{j} \cos(2\pi\alpha f(j))
    = 
 \liminf_{N \to \infty} \inf_{\alpha\in \mathcal{I}_{N}(0)}
\int_{0}^{N} g_N(t) \cos(2\pi\alpha f(t))\,\dt
\end{align}
\noindent Given $\alpha \in \mathbb{R}$ and $N \in \mathbb{N}$, we write $h_{N, \alpha}$ for the function where $h_{N, \alpha}(x)=g_N(x) \cos(2\pi\alpha f(x))$. First, note that for any $\alpha \in \mathbb{R}$ and $N \in \mathbb{N}$, we have\begin{align*}
\left|\sum_{j=1}^N  h_{N, \alpha}(j) - \int_{0}^N  h_{N, \alpha}(t)\,\dt \right| &\le
\sup_{x \in [0,N]}N\left|h_{N, \alpha}'(x)\right|\\
&\le \sup_{x \in [0,N]}N\left|g_{N}'(x)\right|+\sup_{x \in [0,N]}N\left|g_{N}(x)2 \pi\alpha f'(x)\right|.\end{align*}Moreover, $\lim_{N \rightarrow \infty}\sup_{x \in [0,N]}N\left|g_{N}'(x)\right|=0$ by Lemma~\ref{lem:polyDer}. Thus equation~(\ref{eqnPoly}) follows from the fact that, again using Lemma~\ref{lem:polyDer} for the last line,\begin{align*}
\lim_{N \to \infty} \sup_{\alpha\in {\mathcal{I}_{N}(0)}}\sup_{x \in [0,N]}N|g_N(x)2\pi\alpha f'(x)|&=\lim_{N \to \infty}\sup_{x \in [0,N]}N|g_N(x)2\pi N^{-r+\frac{1}{r}}f'(x)|\\
&\le\lim_{N \to \infty} \sup_{x \in [0,N]}\mathcal{O}(N^{\frac{1}{r}})|g_N(x)|\\
&\le\lim_{N \to \infty} \sup_{x \in [0,N]}\mathcal{O}(N^{1+\frac{1}{r}})|g'_N(x)|\\
&=0.
\end{align*}The argument above is actually the only place where we use that $\alpha\in {\mathcal{I}_{N}(0)}$.

Now, recall that since $f$ is positive and increasing on $[0,\infty)$, we know that the denominator $\int_0^N f'(t)\left( 1-\frac{f(t)}{f(N)}\right)\,\dt$ of $g_N$ is always positive and has order $\Omega(N^r)$. So by equation~(\ref{eqnPoly}), it suffices to prove that $\inf_{\alpha\in {\mathcal{I}_{N}(0)}}\int_{0}^{N} f'(t)\left(1-\frac{f(t)}{f(N)} \right)\cos(2\pi\alpha f(t))\,\dt$ is lower-bounded by a negative quantity of absolute value $o(N^r)$. We will show something much stronger: that there exists $C \in \mathbb{R}$ so that for every $N \in \mathbb{N}$ and $\alpha\in \mathbb{R}$, we have $\int_{0}^{N} f'(t)\left(1-\frac{f(t)}{f(N)} \right)\cos(2\pi\alpha f(t))\,\dt \geq C$. This is certainly enough and will complete the proof of the lemma. 

To show this claim, let us fix $N\in\mathbb{N}$ and $\alpha\in\mathbb{R}$ for a moment.
If $\alpha = 0$, then $\cos(2\pi\alpha f(t))=1$ for all $t$ and the considered integral is non-negative, so we may assume that $\alpha\neq 0$. Then
\begin{align*}
\int_{0}^{N} f'(t) \left(  1 - \frac{f(t)}{f(N)}  \right) \cos(2\pi \alpha f(t)) \,\dt &=
\int_{f(0)}^{f(N)}  \left(  1 - \frac{u}{f(N)}  \right) \cos(2\pi \alpha u) \,du.
\end{align*}Moreover, notice that $\left| \int_{0}^{f(0)}\left(  1 - \frac{u}{f(N)}  \right) \cos(2\pi \alpha u)\,du \right| \le  \int_{0}^{f(0)} 1 \,du  = f(0)$, which is a constant. So it suffices to consider the integral whose lower limit is $0$, rather than $f(0)$. We have\begin{align*}
\int_{0}^{f(N)}  &\left(  1 - \frac{u}{f(N)}  \right) \cos(2\pi \alpha u) \,du
\\
&=
\int_{0}^{f(N)} \cos(2\pi \alpha u) \,du
- 
\frac{1}{f(N)}
\int_{0}^{f(N)}
u  \cos(2\pi \alpha u) \,du
\\
&=
\frac{\sin(2\pi \alpha f(N))}{2 \pi \alpha}
-
\left[\frac{\sin(2\pi \alpha f(N))}{2 \pi \alpha}
-
\frac{1}{2\pi \alpha f(N)}
\int_{0}^{f(N)}
  \sin(2\pi \alpha u) \,du\right]
\\
&=
\frac{1- \cos(2\pi \alpha f(N))}{(2\pi \alpha)^2 f(N)} \geq 0.
\end{align*}
This completes the proof of Lemma~\ref{poly zero}.
\end{proof}

Now we use Vinogradov's estimates from Theorem~\ref{poly estimate} to prove the following lemma, which will help take care of exponential sums within the minor arcs.

\begin{lemma}
\label{lem:minor1}
There exists a constant $C >0$ such that for any fixed integer $Q\geq 2$ and any positive integers $N\ge M$ so that $N-M$ is sufficiently large as a function of $Q$, 
\begin{align*}
\sup_{\alpha\in \mathfrak{m}_{N,Q}}
    \left|
    \sum_{j=M+1}^{N}  e(\alpha f(j))
    \right|
 \le C  (N-M)Q^{- \frac{1}{2r} }.
\end{align*}
\end{lemma}
\begin{proof}
    Apply Theorem~\ref{poly estimate} with $\epsilon = \frac{1}{2r}$, and take $C = c(f, \frac{1}{2r})$, where $c$ denotes the function from the theorem. Recall that $\rho_r$ is a fixed real number in the interval $(0,1)$. So we may assume that $N-M \ge Q^{\frac{1}{2r\rho_r}}$. Thus, by Theorem~\ref{poly estimate} part~(\ref{polyEst:Second}), for any $\alpha \in \mathbb{R}$ such that there do not exist coprime integers $a$ and $b$ with $1 \le b \le ({N-M})^{1/r}$ and $|\alpha-\frac{a}{b}| \leq ({N-M})^{-r+\frac{1}{r}}$, we have \begin{align*}
    \left|
    \sum_{j=M+1}^{N}  e(\alpha f(j))
    \right|=\left|
    \sum_{j=1}^{N-M}  e(\alpha f(j+M))
    \right| \leq C(N-M)^{1-\rho_r} \le C (N-M) Q^{-\frac{1}{2r}}.
    \end{align*}
    
    It just remains to consider those $\alpha \in \mathfrak{m}_{N,Q}$ for which there do exist such coprime integers $a$ and $b$. Notice that $b>Q$, since otherwise the interval $\left[ \frac{a}{b} - N^{-r + \frac{1}{r}} , \frac{a}{b} + N^{-r + \frac{1}{r}}  \right]$ is contained in $\mathfrak{M}_{N,Q} \cup \mathcal{I}_{N}$ and therefore is disjoint from $\mathfrak{m}_{N,Q}$. (In fact this argument applies even if $\alpha$ is in the smaller set $\mathfrak{m}_{N-M,Q}$, but we will not need this observation.) So, by Theorem~\ref{poly estimate} part~(\ref{polyEst:First}), \begin{align*}
    \left|
    \sum_{j=M+1}^N  e(\alpha f(j))
    \right| \le C(N-M)b^{-\frac{1}{r}+\epsilon} \le C(N-M)Q^{-\frac{1}{2r}},\end{align*} and the lemma follows.
\end{proof}

For the rest of this section we write $C$ for the constant from Lemma~\ref{lem:minor1}. In our final lemma, we use Lemma~\ref{lem:minor1} to take care of exponential sums within the minor arcs.

\begin{lemma}\label{poly minor}
For any fixed integer $Q \geq 2$,
\[    
\limsup_{N \to \infty} \sup_{\alpha\in \mathfrak{m}_{N,Q}}
    \left|\sum_{j=1}^N 
    \wFunc{j}e(\alpha f(j))
    \right|
 \le CQ^{- \frac{1}{2r} }.
\]
\end{lemma}
\begin{proof}
We evaluate the limit by partitioning the sum into many intervals of large size. Notice that the integers between $1$ and $N$ can be partitioned into $\lfloor\sqrt{N}\rfloor$ intervals each of size at least $\lfloor\sqrt{N}\rfloor$ and at most $2\sqrt{N}$. For any positive integer $i \leq \lfloor\sqrt{N}\rfloor$, we write $N[i]$ for the $i$-th part. So $N[1] = \{1,2,\ldots, \lfloor \sqrt{N}\rfloor\}$, and so on, where every part except for possibly the last one (which can be up to twice as large) has size $\lfloor\sqrt{N}\rfloor$.

We will consider the minimum value of $g_N$ within each interval $N[i]$. For convenience, let us write $m_{N,i}$ for the minimum of $g_N(k)$ over all $k \in N[i]$. This notation will let us split up the sum over $N[i]$ into a positive part (which must be fairly small because the value of $g_N$ cannot change much within an interval), and a remaining part to which we will apply Lemma~\ref{lem:minor1}. More formally, notice that for any $i \in \{1,2,\ldots, \lfloor \sqrt{N}\rfloor\}$ and $\alpha \in \mathbb{R}$, we have \begin{align*}
    \sum_{j\in N[i]} g_{N}(j) e(\alpha f(j))
    =
    \left(\sum_{j\in N[i]} \left( g_{N}(j) - m_{N,i}  \right)e(\alpha f(j))\right)
    +m_{N,i}
    \left(\sum_{j\in N[i]}  e(\alpha f(j))
    \right).
\end{align*}
Since $g_N$ is positive and $|e(\alpha f(j))| \leq 1$ for each $j$, we obtain\begin{align*}
    \left|\sum_{j\in N[i]} g_{N}(j) e(\alpha f(j))\right|
    \le
    \left(\sum_{j\in N[i]} \left( g_{N}(j) - m_{N,i}  \right)\right)
    +m_{N,i}
    \left|\sum_{j\in N[i]}  e(\alpha f(j))
    \right|.
\end{align*}
Thus, for any positive integer $N$ and real number $\alpha$, 
\begin{align*}
    \left|
    \sum_{j=1}^{{N}}  g_{N}(j) e(\alpha f(j))
    \right| 
    &\le
    \sum_{i=1}^{\lfloor\sqrt{N}\rfloor}
    \left|
    \sum_{j\in N[i]} g_{N}(j) e(\alpha f(j))
    \right|
    \\
    & \le
\left(\sum_{i=1}^{\lfloor\sqrt{N}\rfloor}
    \sum_{j\in N[i]} \left( g_{N}(j) - m_{N,i}  \right)\right)
+   
\left(\sum_{i=1}^{\lfloor\sqrt{N}\rfloor}
m_{N,i}
    \left|\sum_{j\in N[i]}  e(\alpha f(j))
    \right|\right).
\end{align*}

We consider these two sums separately, beginning with the one on the left since it does not depend on $\alpha$. Since each part $N[i]$ has size at most $2\sqrt{N}$,  we have \begin{align*}g_{N}(j)-m_{N,i} \leq 2\sqrt{N}\sup_{x \in [0,N]}|g'_N(x)|\qquad \textrm{for any }i\leq \lfloor \sqrt{N} \rfloor\textrm{ and }j\in N[i].\end{align*} Thus, by Lemma~\ref{lem:polyDer}, we have\begin{align*}
    \lim_{N \rightarrow \infty}\left(\sum_{i=1}^{\lfloor\sqrt{N}\rfloor}\sum_{j\in N[i]} \left( g_{N}(j) - m_{N,i}  \right)\right)  \le \lim_{N\rightarrow \infty} \left( 2N^{3/2}\sup_{x \in [0,N]}|g'_N(x)|\right) = 0.
\end{align*}

It just remains to evaluate the second sum. By Lemma~\ref{lem:minor1} and the fact that each part $N[i]$ has size at least $\lfloor\sqrt{N}\rfloor$, 
\begin{align*} 
&\limsup_{N \to \infty}
\sup_{\alpha\in \mathfrak{m}_{N,Q}}
\sum_{i=1}^{\lfloor\sqrt{N}\rfloor}m_{N,i}\left|\sum_{j\in N[i]}  e(\alpha f(j))\right|\\
&\qquad\qquad\le \limsup_{N \to \infty}
\sup_{\alpha\in \mathfrak{m}_{N,Q}}
\sum_{i=1}^{\lfloor\sqrt{N}\rfloor}m_{N,i}\left(C|N[i]|Q^{-\frac{1}{2r}}\right)\\
&\qquad\qquad \le CQ^{-\frac{1}{2r}}\limsup_{N \rightarrow \infty} \sum_{j=1}^{N}g_N(j).
\end{align*}

\noindent Since $\lim_{N \rightarrow \infty}\sum_{j=1}^{N}g_N(j)=1$ by Lemma~\ref{lem:avgPoly}, this completes the proof of Lemma~\ref{poly minor}.
\end{proof}

Finally, we use these exponential sum estimates to bound the independence density in the Cayley graph $G(\mathbb{Z}^2, \pm AD_{Q!, k})$. Recall from the beginning of the section that $A = f(\mathbb{N})$ and, for each positive integer $N$, we write $A_N = \{f(1), f(2), \ldots, f(N)\}$. 

\begin{proposition}\label{poly graphs}
    For any integers $Q,k \geq 2$,
\[    
\overline{\alpha}(G(\mathbb{Z}^2,  \pm A D_{Q!,k})) \le \liminf_{N \rightarrow \infty} \overline{\alpha}(G(\mathbb{Z}^2,  \pm A_N D_{Q!,k}))  \le \frac{4 + 2kCQ^{- \frac{1}{2r} }}{2k+2+ 2kCQ^{- \frac{1}{2r} }}.
\]
\end{proposition}

\begin{proof}
The first inequality holds since $A_1 \subseteq A_2 \subseteq \ldots \subseteq A$. So we just need to upper-bound the limit. For each $N$, we apply Theorem~\ref{ratio bound} to the Cayley graph $G(\mathbb{Z}^2,  \pm A_N D_{Q!,k})$ with weights $w_N$ (and the corresponding function $\widehat{w}_N$) defined as in the beginning of this section. Thus
\[
\overline{\alpha}(G(\mathbb{Z}^2, \pm A_N D_{Q!,k}))
\le
\frac{-\inf_{u\in \mathbb{R}^2}  \widehat{w}_N(u) }{  \sup_{u\in \mathbb{R}^2}  \widehat{w}_N(u)   -\inf_{u\in \mathbb{R}^2}  \widehat{w}_N(u)}.
	\]

We already know that $\lim_{N \rightarrow \infty} \sup_{u\in \mathbb{R}^2}  \widehat{w}_N(u) = 2k-2$ by Lemma~\ref{lem:supPoly}. So it suffices to prove that $\liminf_{N\rightarrow \infty}\inf_{u \in \mathbb{R}^2}\widehat{w}_N(u)\ge -\left( 4+2kCQ^{-\frac{1}{2r}}\right)$. We prove this by splitting the outer sum below into three parts based on whether $u \cdot d$ is in $\mathfrak{M}_{N,Q}$, $\mathcal{I}_N $, or $\mathfrak{m}_{N,Q}$. By considering the three limits separately, and applying Lemmas~\ref{lem:polyMajor},~\ref{poly zero}, and~\ref{poly minor} to the major, integral, and minor arcs (respectively), we obtain\begin{align*}
\liminf_{N \rightarrow \infty}\inf_{u \in \mathbb{R}^2}\widehat{w}_N(u) &= \liminf_{N \rightarrow \infty}\inf_{u \in \mathbb{R}^2}\sum_{d \in \pm D_{Q!, k}}\sum_{j=1}^N g_N(j)e(u \cdot f(j) d) \\
& \ge -4+0-|\pm D_{Q!,k}| C Q^{-\frac{1}{2r}}\\
&\ge -4-2kC Q^{-\frac{1}{2r}}.
\end{align*}

\noindent Here we are using the bound $|\pm D_{Q!, k}|=2k-2\leq 2k$ just to simplify the expression. This completes the proof of Proposition~\ref{poly graphs}.
\end{proof}

Theorem~\ref{poly dist} now follows from Lemma~\ref{embed} and Proposition~\ref{poly graphs} since
\begin{align*}
\limsup_{k \to \infty} \left(\limsup_{Q \to \infty} \chi(G(\mathbb{Z}^2, \pm A D_{Q,k} ))\right)
&\ge 
\limsup_{k \to \infty} \left(\limsup_{Q \to \infty}
\frac{1}{\overline{\alpha}(G(\mathbb{Z}^2, \pm A D_{Q,k} ))}\right)
\\
&\ge
\limsup_{k \to \infty} \left(\limsup_{Q \to \infty}
\frac{2k+2 + 2kCQ^{- \frac{1}{2r} }}{4 + 2kCQ^{- \frac{1}{2r} }}\right)
\\
&=
\limsup_{k \to \infty} 
\frac{2k + 2}{4}
\\
&=
\infty,
\end{align*}
as desired.

\section{Prime distances}\label{sec:prime}

In this section we prove Theorem~\ref{prime dist}. 

First, we briefly recall some basic number theoretic preliminaries.
In this section, $p$ always denotes a prime number; so for instance $\sum_{p\le N} 1$ is equal to the number of primes that are at most $N$.
We let $\mathbb{P}= \{2,3,5,7,11,\ldots \}$ be the set of all prime numbers, and we denote Euler's totient function by $\phi$.
Recall that $\phi(n)= n \Pi_{p | n} \left( 1 - \frac{1}{p} \right)$, and a simple lower bound is given by $\phi(n) \ge \sqrt{n/2}$. (See~\cite[Proposition 2]{lowerBoundTotient} for an elementary proof of this inequality; much stronger asymptotic bounds are also known~\cite[Theorem~328]{hardyWrightBook}.) One form of the prime number theorem, which we shall often apply, is that $\lim_{N \rightarrow \infty}\frac{1}{N}\sum_{p\le N} \log p = 1$. (See~\cite[Section~4.2]{ApostolBook}.) Here and throughout this section, all logarithms are the natural logarithm. 

As with the proof of Theorem~\ref{poly dist}, we must examine some exponential sums in order to apply Theorem~\ref{ratio bound}, and examining these exponential sums splits into three cases given by integral, major, and minor arcs. However, in this section we define those arcs differently. So, let $N\in \mathbb{N}$.
For every $a\in \mathbb{Z}$, we define the \emph{integral arc} $\mathcal{I}_{N}(a) = \left[a-\frac{(\log N)^9}{N}, a + \frac{(\log N)^9}{N} \right]$. We write $\mathcal{I}_N = \bigcup_{a\in \mathbb{Z}} \mathcal{I}_{N}(a)$ for the union of all integral arcs.
Let $Q\ge 2$ be an integer.
For every pair of non-zero coprime integers $a$ and $b$ with $2\le b \le Q$, we define the \emph{major arc}
\[
\mathfrak{M}_{N,Q}(a,b) = \left[ \frac{a}{b}-\frac{(\log N)^9}{N} , \frac{a}{b} + \frac{(\log N)^9}{N} \right].
\]
Again, we write $\mathfrak{M}_{N,Q}$ for the union of the above majors arcs. We call any interval of $\mathbb{R}$ which is disjoint from $\mathcal{I}_N$ and $\mathfrak{M}_{N,Q}$ a \emph{minor arc}, and we write $\mathfrak{m}_{N,Q}= \mathbb{R} \backslash (\mathfrak{M}_{N,Q} \cup \mathcal{I}_N)$.

Let us outline our strategy. Given integers $Q,k\ge 2$, we will bound the independence density of the Cayley graph $G(\mathbb{Z}^2, \pm \mathbb{P}D_{Q!, k})$. This will be enough because, by Lemma~\ref{embed}, $G(\mathbb{Z}^2, \pm \mathbb{P}D_{Q!, k})$ is a subgraph of the desired distance graph of the plane. As in the last section, we will bound this independence density by, for larger and larger integers $N$, applying Theorem~\ref{ratio bound} to the Cayley graph $G(\mathbb{Z}^2, \pm \mathbb{P}_ND_{Q!, k})$ with weights $w_N$, where $\mathbb{P}_N$ is the set of all primes $p \leq N$, and $w_N$ is the function on $\pm \mathbb{P}_ND_{Q!, k}$ defined as follows. 

If $Q$ and $k$ are clear from context, then for each prime $p \leq N$ and $d \in \pm D_{Q!, k}$, we write \begin{align*}
w_N(pd) = \frac{1}{N}\left(1-\frac{p}{N} \right) \log{p}.
\end{align*}
\noindent Note that $w_N$ is well-defined since no two elements in $D_{Q!, k}$ are scalar multiples of each other by Lemma~\ref{not scalar multiples}. Moreover, the value of $w_N(pd)$ only depends on~$p$. So it is convenient to define a similar function but on $\mathbb{N}$. Thus, for any $j \in \mathbb{N}$, we write \begin{align*}\wFunc{j} = \frac{1}{N}\left(1-\frac{j}{N} \right) \log{j}.\end{align*}

Recall that, in order to apply Theorem~\ref{ratio bound}, we need to bound the infimum and supremum of the function $\widehat{w}_N$ whose value, given $u\in \mathbb{R}^2$, is the sum over all $x \in \pm \mathbb{P}_ND_{Q!, k}$ of $w_N(x)e(u\cdot x)$. Re-writing this sum, we have\begin{align*}
 \widehat{w}_N(u) = \sum_{d \in \pm D_{Q!, k}}\sum_{p \leq N} g_N(p)e(u\cdot pd).
\end{align*}
\noindent As before, we often bound $\widehat{w}_N(u)$ by fixing $d$ and considering the inner sum.

We begin by evaluating the following useful sum.

\begin{lemma}
\label{lem:avgPrime} 
We have\begin{align*}\lim_{N \rightarrow \infty} \sum_{p \le N} g_N(p)=\frac{1}{2}.
\end{align*}
\end{lemma}
\begin{proof}
    We prove this inequality by breaking the sum up into larger and larger pieces, and considering how many pieces each prime occurs in. So, let $N,M\in \mathbb{N}$ with $M \le N$, and consider a prime $p \le N$. Notice that $p$ occurs in $M-\left\lceil\frac{Mp}{N} \right\rceil$ of the intervals $\{[0,\frac{Nj}{M}]: j \in \{1,2,\ldots, M-1\}\}$. Moreover, since $\left|\left\lceil\frac{Mp}{N} \right\rceil -\frac{Mp}{N}\right|\le 1$, we have\begin{align*}
    \left| \frac{1}{M-\left\lceil\frac{Mp}{N} \right\rceil} - \frac{1}{M\left(1-\frac{p}{N} \right)}\right| \leq \frac{1}{M\left(1-\frac{p}{N} \right)}.
    \end{align*}Thus \begin{align*}
   \left|\sum_{p \le N} g_N(p)-\sum_{j=1}^{M-1} \sum_{p \leq \frac{Nj}{M}}\frac{g_N(p)}{M\left( 1-\frac{p}{N}\right)}\right| \le \sum_{p \le N} \frac{g_N(p)}{M\left( 1-\frac{p}{N}\right)}=\sum_{p \le N} \frac{\log{p}}{NM}.
    \end{align*}By the prime number theorem, $\lim_{M \rightarrow \infty}\lim_{N \rightarrow \infty}\sum_{p \le N} \frac{\log{p}}{NM} = \lim_{M \rightarrow \infty} \frac{1}{M}=0$. So it suffices to consider the limit of the double sum above. Again by the prime number theorem,\begin{align*}
        \lim_{M \rightarrow \infty} \lim_{N \rightarrow \infty}\sum_{j=1}^{M-1} \sum_{p \leq \frac{Nj}{M}}\frac{g_N(p)}{M\left( 1-\frac{p}{N}\right)}
        =\lim_{M \rightarrow \infty} \lim_{N \rightarrow \infty}\sum_{j=1}^{M-1} \frac{j}{M^2}\sum_{p \leq \frac{Nj}{M}} \frac{M\log{p}}{Nj}
        = \lim_{M \rightarrow \infty}\sum_{j=1}^{M-1} \frac{j}{M^2} = \frac{1}{2}.
    \end{align*}This completes the proof of Lemma~\ref{lem:avgPrime}.
\end{proof}

We are now ready to compute the supremum.

\begin{lemma}
\label{lem:supPrime}
For any fixed integers $Q,k \geq 2$, \begin{align*}\lim_{N \rightarrow \infty} \sup_{u \in \mathbb{R}^2} \widehat{w}_N(u) = k-1.
\end{align*}
\end{lemma}
\begin{proof}
	Recall from Section~\ref{sec:lattices} that the supremum is attained at $\widehat{w}_N(0)=\sum_{d \in \pm D_{Q!, k}}\sum_{p \le N} g_N(p)$. Thus we are done by Lemma~\ref{lem:avgPrime} and the fact that  $|\pm D_{Q!, k}|=2|D_{Q!,k}|=2(k-1)$.
\end{proof}

Next we use Lemma~\ref{sandwich} to handle the exponential sums within the major arcs.

\begin{lemma}
\label{lem:primeMajor} For any fixed integers $Q, k \geq 2$,\begin{align*}
    \liminf_{N \rightarrow \infty} \inf_{u \in \mathbb{R}^2} \sum_{\substack{d \in \pm D_{Q!, k}\\u \cdot d \in \mathfrak{M}_{N,Q}}}\sum_{p \le N} g_N(p)e(u \cdot pd) \geq -2.
\end{align*}
\end{lemma}
\begin{proof}
    The proof of this lemma is essentially the same as for Lemma~\ref{lem:polyMajor} in the last section. First of all, by Lemma~\ref{sandwich}, there exists $\epsilon_{Q,k}>0$ such that for every $u\in \mathbb{R}^2$, there are at most four elements $d\in \pm D_{Q!,k}$ such that $\left|u \cdot d - \frac{a}{b} \right| < \epsilon_{Q,k}$ for some non-zero coprime integers $a$ and $b$ with $2\le b\le Q$. Thus, for any sufficiently large integer $N$, there are at most four elements $d\in \pm D_{Q!,k}$ such that $u \cdot d \in \mathfrak{M}_{N,Q}$. Moreover, again by Lemma~\ref{sandwich}, the set of all such $d$ is centrally symmetric. The lemma then follows from the fact that $e(\alpha)+e(-\alpha)=2\cos(2\pi\alpha)\geq -2$ for any real number $\alpha$, and the fact that $\lim_{N \rightarrow \infty}\sum_{p \le N}^N g_N(p) =\frac{1}{2}$ by Lemma~\ref{lem:avgPrime}.
\end{proof}

Next we shall examine certain exponential sums within the integral arcs.
We require the following more precise version of the prime number theorem that includes the first few lower order terms \cite{poussin1899function}.

\begin{theorem}[de La~Vall{\'e}e Poussin \cite{poussin1899function}]\label{pnt}
    We have that
    \[
    \sum_{p\le N} \log p = \sum_{k=0}^{10} \frac{k!N}{(\log N)^k} + \mathcal{O}\left(\frac{N}{(\log N)^{11}}\right).
    \]
\end{theorem}

As a simple corollary of Theorem \ref{pnt}, we obtain the following. We note that the big-$\mathcal{O}$ notation in the statement below does \emph{not} consider $m$ to be fixed.

\begin{corollary}\label{pntinterval}
    For any integer $m$ with $0 \le m \le (\log N)^{10}$, we have that
    \[
    \sum_{m \frac{N}{( \log N )^{10}  }< p \le (m+1)\frac{N}{ ( \log N )^{10} }} \log p = \frac{N}{(\log N)^{10}} + \mathcal{O}\left(\frac{N}{(\log N)^{11}}\right).
    \]
\end{corollary}
\begin{proof}
First note that when $m=0$, the corollary follows by applying Theorem~\ref{pnt} to the integer $N_2 = \lceil\frac{N}{( \log N )^{10}}\rceil$. The term for $k=0$ is $\lceil\frac{N}{(\log N)^{10}}\rceil$, and each of the other terms is $\mathcal{O}\left(\frac{N}{(\log N)^{11}}\right)$. So the corollary holds in this case, and we may assume that $m \geq 1$.

Now, apply Theorem~\ref{pnt} to the integers $N_1 =  m \lceil\frac{N}{( \log N )^{10}}\rceil$ and $N_2= (m+1) \lceil\frac{N}{( \log N )^{10}}\rceil$. Note that $|N_1-N_2|=\lceil\frac{N}{( \log N )^{10}}\rceil$, and consider, for a fixed positive integer $k$, how much the function $f_k(x) = \frac{x}{(\log x)^k}$ can change between $N_1$ and $N_2$. Since $N_1 \le N_2$, we have that $\sup_{x \in [N_1,N_2]}f'_k(x) = \mathcal{O}(\frac{1}{\log{N_1}}) = \mathcal{O}(\frac{1}{\log{N}})$ since $m\geq 1$. Thus $|f_k(N_1)-f_k(N_2)| = \mathcal{O}(\frac{N}{(\log{N})^{11}})$ for each positive integer $k$, which completes the proof.
\end{proof}

We are now ready to examine the exponential sums within the integral arcs.

\begin{lemma}\label{prime zero}
    We have that
    \[
    \liminf_{N \to \infty} \inf_{\alpha\in \mathcal{I}_N}
    \sum_{p \le N} 
    g_N(p)\left[e(\alpha p) + e(-\alpha p)\right] \ge 0.
    \]
\end{lemma}

\begin{proof}
As in the proof of Lemma~\ref{poly zero}, it is convenient to write $e(\alpha p) + e(-\alpha p)=2\cos(2\pi\alpha p)$. Since this value is the same for $\alpha$ and $\alpha+1$, it makes no difference to take the infimum over all $\alpha \in \mathcal{I}_{N}(0)$ instead of over all $\alpha \in \mathcal{I}_{N}$. Consequently, it is enough to prove that $\liminf_{N \to \infty} \inf_{\alpha\in {\mathcal{I}_{N}(0)}} \sum_{p \le N} g_N(p) \cos(2\pi\alpha p) \geq 0$. 
If $\alpha = 0$, then this sum is non-negative since $g_N(p) \ge 0$ for every $p \le N$. So we may assume that $\alpha \neq 0$.

We now prove inequality~(\ref{eqnPoly2}) below. After that, we will evaluate its right-hand side.
\begin{align}
\label{eqnPoly2}
    \liminf_{N \to \infty} \inf_{\alpha\in {\mathcal{I}_{N}(0)}} \sum_{p\le N} \wFunc{p} \cos(2\pi\alpha p)
    \geq 
 \liminf_{N \to \infty} \inf_{\alpha\in \mathbb{R}\setminus \{0\}}
\frac{1}{N}\int_{0}^{N} \left(1- \frac{t}{N}\right) \cos(\alpha t)\,\dt
\end{align}
To prove that this inequality holds, we shall partition the sum into roughly $(\log N)^{10}$ intervals of size about $N/(\log N)^{10}$. More precisely, the integers between $1$ and $N$ can be partitioned into $\lfloor(\log N)^{10} \rfloor$ intervals each of size $\lfloor N/(\log N)^{10} \rfloor$ or $\lceil N/(\log N)^{10} \rceil$. For any positive integer $m \leq (\log N)^{10}$, we write $N[m]$ for the $m$-th part. So $N[1] = \{1,2,\ldots, \lfloor N/(\log N)^{10} \rfloor\}$, and so on, where every part has size either $\lfloor N/(\log N)^{10} \rfloor$ or $\lceil N/(\log N)^{10} \rceil$. 

Now, for any positive integer $N$ and real number $\alpha$, 
\[
\sum_{p \le N} g_N(p) \cos(2\pi\alpha p)
=
\sum_{m=1}^{(\log N)^{10} }
\sum_{p\in N[m]} g_N(p) \cos(2\pi\alpha p).
\]Informally, our goal is to apply Corollary~\ref{pntinterval} to the $\log p$ part of $g_N(p)=\frac{1}{N}\left(1-\frac{p}{N} \right)\log p$ within the inner sum above. Then we will prove that the resulting function is close to a corresponding integral. Towards this end, it is convenient to define the following function. For any fixed non-zero $\alpha \in \mathbb{R}$ and $N \in \mathbb{N}$, we write $h_{N, \alpha}$ for the function where \begin{align*}
h_{N, \alpha}(x) = \left(
    1 - \frac{x}{( \log N )^{10}} 
    \right) 
    \cos\left(2\pi\alpha x \frac{N}{ ( \log N )^{10} }\right).\end{align*}

For each $1\le m \le (\log N)^{10}$, we have $\lim_{N\to \infty} N(\max_{j\in N[m]} g_N(j) - \min_{j\in N[m]} g_N(j)) \le \lim_{N\to \infty} (\log N)^{-9} =0$, so $\lim_{N\to \infty} N(\max_{j\in N[m]} g_N(j) - \min_{j\in N[m]} g_N(j))=0$ since it is non-negative.
We also have that $\limsup_{N\to \infty} \sup_{\alpha\in \mathcal{I}_N(0)} \left|\alpha N / (\log N)^{10} \right| \le \lim_{N\to \infty} (\log N)^{-1}  = 0$ since $|\alpha|\le (\log N)^9/N$.
Therefore,
$\lim_{N\to \infty} (\cos(2\pi\alpha j) - \min_{j\in N[m]} \cos(2\pi\alpha j) )=0$, and furthermore,
$\lim_{N\to \infty} N(\max_{j\in N[m]} g_N(j)\cos(2\pi\alpha j) - \min_{j\in N[m]} g_N(j) \cos(2\pi\alpha j) )=0$.
Hence,
\begin{align*}
\liminf_{N \to \infty} &  \inf_{\alpha\in \mathcal{I}_N(0)}
\sum_{m=1}^{(\log N)^{10} }
\sum_{p\in N[m]} g_N(p) \cos(2\pi\alpha p)
=
\liminf_{N \to \infty}  \inf_{\alpha\in \mathcal{I}_N(0)}
\frac{1}{N}
\sum_{m=1}^{(\log N)^{10} }
\sum_{p\in N[m]}h_{N, \alpha}(m) \log p.
\end{align*}
Then, by Corollary~\ref{pntinterval}, 
\begin{align*}
    \liminf_{N \to \infty}  \inf_{\alpha\in \mathcal{I}_N(0)}
\frac{1}{N}
\sum_{m=1}^{(\log N)^{10} }
\sum_{p\in N[m]}h_{N, \alpha}(m)(\log p) 
    &=
    \liminf_{N \to \infty}  \inf_{\alpha\in \mathcal{I}_N(0)}
\frac{1}{( \log N )^{10}}
\sum_{m=1}^{(\log N)^{10} }
h_{N, \alpha}(m).
\end{align*}
Now, for any fixed non-zero $\alpha \in \mathbb{R}$ and $N \in \mathbb{N}$, note that
    \begin{align*}
    \left| \frac{1}{(\log N)^{10}}\sum_{m=1}^{( \log N )^{10}}h_{N, \alpha}(m) - \frac{1}{(\log N)^{10}}\int_{0}^{(\log N )^{10}}h_{N, \alpha}(t)\,\dt\right| 
    &\le \sup_{x \in [0,(\log N )^{10}]} |h_{N, \alpha}'(x)|.
    \end{align*}

    By evaluating this derivative, it can be verified that this supremum is at most $\frac{4\pi|\alpha| N +1}{(\log{N})^{10}}$. Since $\alpha\in \mathcal{I}_{N}(0) = \left[-\frac{(\log N)^9}{N}, \frac{(\log N)^9}{N}\right]$, putting this all together we have proven that\begin{align*}
    \liminf_{N \to \infty} \inf_{\alpha\in {\mathcal{I}_{N}(0)}} \sum_{p\le N} \wFunc{p} \cos(2\pi\alpha p)
    =
 \liminf_{N \to \infty} \inf_{\alpha\in {\mathcal{I}_{N}(0)}} 
\frac{1}{(\log N )^{10}}\int_{0}^{(\log N )^{10}} h_{N,\alpha}(t)\,\dt.
\end{align*}
This completes the proof of equation~(\ref{eqnPoly2}) since $(\log{N})^{10}\rightarrow \infty$ as $N\rightarrow \infty$.

Finally, for any non-zero $\alpha \in \mathbb{R}$, we have
\begin{align*}
\frac{1}{N}\int_{0}^{N} \left(1- \frac{t}{N}\right) \cos(\alpha t)\,\dt &= 
\frac{1}{N}\int_{0}^{N} \cos(\alpha t) \,\dt 
-
\frac{1}{N^2}\int_{0}^{N} 
t \cos(\alpha t) \,\dt
\\&=\frac{1}{N}\left[\frac{\sin(\alpha x)}{\alpha }\right]_0^N-\frac{1}{N^2}\left[\frac{x\sin(\alpha x)}{\alpha }+\frac{\cos(\alpha x)}{\alpha^2}\right]_0^N
\\
&=-\frac{1}{N^2}\left[\frac{\cos(\alpha x)}{\alpha^2}\right]_0^N\\
&=\frac{1}{N^2}\left(\frac{1- \cos(\alpha N)}{\alpha^2}\right).
\end{align*}
The lemma now follows since $\frac{1- \cos(\alpha N)}{\alpha^2} \ge 0$ for all non-zero $\alpha\in \mathbb{R}$ and all $N \in \mathbb{N}$.
\end{proof}

The next step is to examine the exponential sums within the minor arcs. We require Dirichlet's theorem (see~\cite[Theorem~4.1]{nathanson1996additive}), which says the following.

\begin{theorem}[Dirichlet]\label{dirichlet}
    For any real numbers $\alpha$ and $W$ with $W \ge 1$, there exist coprime integers $a$ and $b$ with $1\leq b\le W$ and $\left| \alpha - \frac{a}{b} \right| \le \frac{1}{bW}$.
\end{theorem}

This theorem helps us apply the following estimates of Vinogradov~\cite{vinogradov1937representation,vinogradow1937some} for exponential sums with prime exponents.
These estimates were originally used to prove that every sufficiently large odd number is the sum of three primes.
See also~\cite[Lemma~8.3]{nathanson1996additive} for a presentation of this result.

\begin{theorem}[Vinogradov {\cite{vinogradov1937representation,vinogradow1937some}}]\label{prime estimate 1}
    For any positive real numbers $B$ and $C$ so that $C > 2B$, there exists a real number $\hat{c}(B,C)>0$ satisfying the following. Let $\alpha \in \mathbb{R}$ and $N \in \mathbb{N}$ be such that there exist coprime integers $a$ and $b$ with $1 \le b\le (\log N)^B$ and $\left| \alpha - \frac{a}{b} \right| \le \frac{(\log N)^B}{N}$. Then
    \[
    \left| \sum_{p\le N} (\log p) e(\alpha p)\right|\le  \frac{N}{\phi(b)}+\hat{c}(B,C)\left(\frac{N(\log N)^{2B}}{(\log N)^C} \right).
    \]
\end{theorem}

We also require the following bound, which can be found in~\cite[Theorem~8.5]{nathanson1996additive} as well.
\begin{theorem}[Vinogradov {\cite{vinogradov1937representation,vinogradow1937some}}]\label{prime estimate 2}
    There exists $c>0$ satisfying the following.
    Let $\alpha \in \mathbb{R}$ and $N\in \mathbb{N}$ be such that there exist coprime integers $a$ and $b$ with $1 \le b\le N$ and $\left| \alpha - \frac{a}{b} \right| \le \frac{1}{b^2}$. Then
    \[
    \left|
    \sum_{p\le N} (\log p) e(\alpha p)
    \right|
    \le c \left( \frac{N}{b^{\frac{1}{2}}} + N^{\frac{4}{5}}   +  N^{\frac{1}{2}}b^{\frac{1}{2}}                   \right) (\log N)^4.
    \]
\end{theorem}

By gathering all the tools recalled above, we now obtain the following more unified estimate for certain exponential sums within the minor arcs.

\begin{lemma}\label{prime minor 1}
For any fixed integer $Q \ge 2$, we have
\[    
\limsup_{N \to \infty} \sup_{\alpha\in \mathfrak{m}_{N,Q}}  \left| \frac{1}{N} \sum_{p\le N} (\log p) e(\alpha p) \right| \le  \max_{r>Q} \left\{ \frac{1}{\phi(r)} \right\}.
\]
\end{lemma}

\begin{proof}
Consider an integer $N$ which is sufficiently large so that $N(\log N)^{-9} \ge 1$. By Theorem~\ref{dirichlet} applied with $W = N(\log N)^{-9}$, we have that for each $\alpha \in \mathbb{R}$, there exist coprime integers $a$ and $b$ with $1\le b \le  N(\log N)^{-9}$ and $\left| \alpha - \frac{a}{b} \right| \le \frac{(\log N)^{9}}{bN}$.
Moreover, if $\alpha \in \mathfrak{m}_{N,Q}$, then we further have that $Q<b$.

First, if $b \le (\log N)^9$, then by Theorem~\ref{prime estimate 1} applied with $B=9$ and $C=19$, we have that
\[
\left|\sum_{p\le N} (\log p) e(\alpha p)\right| \leq \frac{N}{\phi(b)}+\hat{c}(B,C)\left(\frac{N}{\log N} \right),
\]where $\hat{c}(B,C)$ denotes the constant from the theorem.
Note that
\[
\lim_{N \to \infty} \frac{1}{N} \left( \frac{N}{\phi(b)}  + \hat{c}(B,C)\left(\frac{N}{\log N} \right) \right)
= \frac{1}{\phi(b)}\le  \max_{r>Q} \left\{ \frac{1}{\phi(r)} \right\}.
\] 

For the second case, suppose that $(\log N)^9 < b \le N(\log N)^{-9}$. First we apply Theorem~\ref{prime estimate 2}, where we write $c$ for the universal constant from this theorem; then we use the fact that $b$ is relatively large. So, since $\left| \alpha - \frac{a}{b} \right| \le \frac{(\log{N})^9}{bN} \le \frac{1}{b^2}$, we have
\begin{align*}
\left|
 \sum_{p\le N} (\log p) e(\alpha p)
 \right|
&\le
 c \left( \frac{N}{b^{\frac{1}{2}}} + N^{\frac{4}{5}}   +  N^{\frac{1}{2}}b^{\frac{1}{2}}                   \right) (\log N)^4 
\\
&\le
c \left( N (\log N)^{- \frac{9}{2}} + N^{\frac{4}{5}} + N(\log N)^{- \frac{9}{2}} \right) (\log N)^4
\\
&\le
c \left( N (\log N)^{- \frac{1}{2}} + N^{\frac{4}{5}} (\log N)^4 + N(\log N)^{- \frac{1}{2}} \right).
\end{align*}
Note that 
\[
\lim_{N \to \infty} \frac{1}{N} c \left( N (\log N)^{- \frac{1}{2}} + N^{\frac{4}{5}} (\log N)^4 + N(\log N)^{- \frac{1}{2}} \right)
= 0.
\]Lemma~\ref{prime minor 1} now follows by combining these two cases.
\end{proof}

We are now ready to obtain the minor arc estimate for the actual exponential sum we shall require.

\begin{lemma}\label{prime minor 2}
For any fixed integer $Q \ge 2$, we have
\[    
\limsup_{N \to \infty} \sup_{\alpha\in \mathfrak{m}_{N,Q}}  \left| \sum_{p\le N} g_N(p) e(\alpha p) \right| \le  \frac{1}{2} \max_{r>Q} \left\{ \frac{1}{\phi(r)} \right\}.
\]
\end{lemma}

\begin{proof}
The proof of this lemma is similar to the proof of Lemma~\ref{lem:avgPrime}. We break the sum up into larger and larger pieces, and consider how many pieces each prime occurs in. So, let $N,M\in \mathbb{N}$ with $M \le N$, and consider a prime $p \le N$. Notice that $p$ occurs in $M-\left\lceil\frac{Mp}{N} \right\rceil$ of the intervals $\{[0,\frac{Nj}{M}]: j \in \{1,2,\ldots, M-1\}\}$. Moreover, we have\begin{align*}
\left| \frac{1}{M-\left\lceil\frac{Mp}{N} \right\rceil}-\frac{1}{M\left(1-\frac{p}{N} \right)}\right| \le \frac{1}{M\left(1-\frac{p}{N} \right)}.
\end{align*}Thus, by applying the triangle inequality once per prime $p \le N$, for any $\alpha \in \mathbb{R}$ we have \begin{align*}
  \left| \sum_{p\le N} g_N(p) e(\alpha p) \right| \le \left|  \sum_{j=1}^{M-1}\sum_{p \le \frac{Nj}{M}}\frac{g_N(p) e(\alpha p)}{M\left(1-\frac{p}{N} \right)}\right|+ \sum_{p \le N}\frac{g_N(p)}{M\left(1-\frac{p}{N} \right)}.
\end{align*}The rightmost sum does not depend on $\alpha$, and by the prime number theorem we have\begin{align*}
   \lim_{N\rightarrow \infty}\sum_{p \le N}\frac{g_N(p)}{M\left(1-\frac{p}{N} \right)} = \lim_{N\rightarrow \infty}\sum_{p \le N}\frac{\log{p}}{NM} = \frac{1}{M}.
\end{align*}For the other sum, by Lemma~\ref{prime minor 1}, for each $M\in\mathbb{N}$ and $j \in \{1,2,\ldots, M-1\}$, we have\begin{align*}\limsup_{N\rightarrow \infty}\sup_{\alpha\in \mathfrak{m}_{N,Q}} \left|  \sum_{p \le \frac{Nj}{M}}\frac{g_N(p) e(\alpha p)}{M\left(1-\frac{p}{N} \right)}\right| &= \frac{j}{M^2}\left(\limsup_{N\rightarrow \infty}\sup_{\alpha\in \mathfrak{m}_{N,Q}} \left|  \frac{M}{Nj}\sum_{p \le \frac{Nj}{M}}(\log{p}) e(\alpha p)\right|\right)\\
&\leq \frac{j}{M^2}\left(\max_{r>Q} \left\{ \frac{1}{\phi(r)} \right\}\right).\end{align*} Lemma~\ref{prime minor 2} now follows from the fact that $\lim_{M \rightarrow \infty}\frac{1}{M} = 0$ and $\lim_{M \rightarrow \infty}\sum_{j=1}^{M-1}\frac{j}{M^2} = \frac{1}{2}$.
\end{proof}

We are now ready to use these exponential sum estimates to bound the independence density of the Cayley graph $G(\mathbb{Z}^2, \pm\mathbb{P}D_{Q!, k})$. Recall that for each $N\in \mathbb{N}$, we write $\mathbb{P}_N$ for the set of all primes $p \le N$.

\begin{proposition}\label{prime graphs}
    For any integers $Q,k\ge 2$,
\[    
\overline{\alpha}(G(\mathbb{Z}^2, \pm \mathbb{P}D_{Q!,k})) \le \liminf_{N \rightarrow \infty} \overline{\alpha}(G(\mathbb{Z}^2, \pm \mathbb{P}_N D_{Q!,k}))\le\frac{2 + k\max_{r>Q} \left\{ \frac{1}{\phi(r)} \right\}}{k + 1+k\max_{r>Q} \left\{ \frac{1}{\phi(r)} \right\}}.
\]
\end{proposition}

\begin{proof}
The first inequality holds since $\mathbb{P}_1 \subseteq \mathbb{P}_2 \subseteq \ldots \subseteq \mathbb{P}$. So we just need to upper-bound the limit.
For each $N$, we apply Theorem~\ref{ratio bound} to the Cayley graph $G(\mathbb{Z}^2, \pm \mathbb{P}_N D_{Q!,k})$ with weights $w_N$ (and the corresponding weight function $\widehat{w}_N$) defined at the beginning of this section. Thus\[
\overline{\alpha}(G(\mathbb{Z}^2, \pm \mathbb{P}_N D_{Q!,k}))
\le
\frac{-\inf_{u\in \mathbb{R}^2}  \widehat{w}_N(u) }{  \sup_{u\in \mathbb{R}^2}  \widehat{w}_N(u)   -\inf_{u\in \mathbb{R}^2}  \widehat{w}_N(u)}.
	\]

Lemma~\ref{lem:supPrime} tells us right away that $\lim_{N \rightarrow \infty}\sup_{u\in \mathbb{R}^2} \widehat{w}_N(u)=k-1$. So it suffices to prove that $\liminf_{N \rightarrow \infty}\inf_{u\in \mathbb{R}^2} \widehat{w}_N(u) \ge - \left( 2 + k\max_{r>Q} \left\{ \frac{1}{\phi(r)} \right\}\right)$. We do so by splitting the outer sum below into three parts based on whether $u \cdot d$ is in $\mathfrak{M}_{N,Q}$, $\mathcal{I}_N $, or $\mathfrak{m}_{N,Q}$. By considering the three limits separately, and applying Lemmas~\ref{lem:primeMajor},~\ref{prime zero}, and~\ref{prime minor 2} to the major, integral, and minor arcs (respectively), we obtain\begin{align*}
\liminf_{N \rightarrow \infty}\inf_{u\in \mathbb{R}^2} \widehat{w}_N(u) &\ge \liminf_{N \rightarrow \infty}\inf_{u\in \mathbb{R}^2} \sum_{d \in \pm D_{Q!,k}} \sum_{p \le N} g_N(p) e(u\cdot pn)\\
& \ge -\left(2 +0 +\frac{1}{2}|\pm D_{Q!,k}|\max_{r>Q} \left\{ \frac{1}{\phi(r)} \right\}\right)\\
& \ge -\left(2 +k\max_{r>Q} \left\{ \frac{1}{\phi(r)} \right\}\right)
\end{align*}
Here we are using the bound $|\pm D_{Q!, k}|=2k-2\leq 2k$ just to simplify the expression. This completes the proof of Proposition~\ref{prime graphs}.
\end{proof}

Theorem~\ref{prime dist} now follows from Lemma~\ref{embed} and Proposition~\ref{prime graphs} since
\begin{align*}
\limsup_{k \to \infty} \lim_{Q \to \infty} \chi(G(\mathbb{Z}^2, \pm \mathbb{P} D_{Q!,k}))
&\ge 
\limsup_{k \to \infty} \lim_{Q \to \infty}
\frac{1}{\overline{\alpha}(G(\mathbb{Z}^2, \pm \mathbb{P} D_{Q!,k}))}
\\
&\ge
\limsup_{k \to \infty} \lim_{Q \to \infty}
\frac{k + 1+k\max_{r>Q} \left\{ \frac{1}{\phi(r)} \right\}}{2 + k \max_{r>Q} \left\{ \frac{1}{\phi(r)} \right\}}
\\
&\ge
\limsup_{k \to \infty}
\frac{k + 1 }{2}
\\
&=
\infty ,
\end{align*}
where we are using the bound that $\phi(r) \ge \sqrt{r/2}$ for each $r \in \mathbb{N}$.

\section{Concluding remarks and open problems}\label{sec:con}

The Furstenberg-S{\'a}rk{\"o}zy theorem~\cite{furstenberg1977ergodic,sarkozy1978difference} about colourings of the integers can be extended by the polynomial van der Waerden’s theorem of Bergelson and Leibman~\cite{bergelson1996polynomial}. As a special case, this theorem states that for every non-zero polynomial $f$ with integer coefficients so that $f(0)=0$, and for every finite colouring of the integers, there exist arbitrarily long monochromatic arithmetic progressions with difference $f(n)$ for some $n \in \mathbb{Z}$. 

Bergelson and Leibman~\cite{bergelson1996polynomial} more generally proved a multidimensional version.
A set $X'\subseteq \mathbb{R}^d$ is a \emph{homothetic copy} of a set $X\subseteq \mathbb{R}^d$ if $X'$ can be obtained from $X$ by the composition of a (positive) uniform
scaling and a translation.

\begin{theorem}[Bergelson and Leibman  {\cite{bergelson1996polynomial}}]\label{polyvan}
    Let $f(x)=a_rx^r + \cdots + a_1x$ be a polynomial with integer coefficients so that $a_r \ge 1$ and $f(0)=0$, and let $X\subseteq \mathbb{Z}^d$ be a finite set.
    Then every finite colouring of $\mathbb{Z}^d$ contains a monochromatic homothetic copy of $X$ with scaling factor $f(n)$ for some $n\in \mathbb{N}$ with $f(n)\ge 1$.
\end{theorem}

Gallai (see~\cite{rado1945note} for a published proof) showed that for every finite set $X\subset \mathbb{R}^d$, every finite colouring of $\mathbb{R}^d$ contains a monochromatic homothetic copy of $X$. A standard proof of Gallai's theorem uses the multi-dimensional van der Waerden’s theorem.
It is well know that with exactly the same standard transference, one can obtain the following from Theorem~\ref{polyvan}.
For another version of Gallai's theorem with restrictions, see~\cite{davies2021solution}.

\begin{theorem}[Bergelson and Leibman  {\cite{bergelson1996polynomial}}]
\label{thm:multiDim}
    Let $f(x)=a_rx^r + \cdots + a_1x$ be a polynomial with integer coefficients so that $a_r \ge 1$ and $f(0)=0$, and let $X\subseteq \mathbb{R}^d$ be a finite set.
    Then every finite colouring of $\mathbb{R}^d$ contains a monochromatic homothetic copy of $X$ with scaling factor $f(n)$ for some $n\in \mathbb{N}$ with $f(n)\ge 1$.
\end{theorem}

One might hope that Theorem~\ref{thm:multiDim} could be extended to more general polynomials (or even primes) along the lines of Theorems~\ref{poly dist} and~\ref{prime dist}: that even for polynomials like $f(n)=2n+1$, every finite colouring of $\mathbb{R}^d$ contains a monochromatic homothetic $M$ copy of $f(n)X$ for each finite set $X\subset \mathbb{R}^d$ (in other words, some isometry maps $f(n)X$ to $M$).
However, this turns out to be false. We prove that even for sets $X\subset \mathbb{Z}^d$, the following condition is necessary: for every non-zero $k\in \mathbb{Z}$, we have $f(\mathbb{Z}) \cap k \mathbb{Z} \not= \emptyset$. 

Our proof is based on spherical colourings as considered by Erd{\H{o}}s, Graham, Montgomery, Rothschild, Spencer, and Straus~\cite{erdos1973euclidean}.
For each positive integer $d$, we let $L^d$ be a three point line in $\mathbb{Z}^d$ consisting of points $x,y,z$ with $\|x-y\|_2=\|y-z\|_2=1$ and $\|x-z\|_2=2$.

\begin{proposition}\label{line}
    Let $R\subseteq \mathbb{N}$ be such that there is a positive integer $k$ with $R\cap k\mathbb{Z} = \emptyset$.
    Then for any positive integer $d$, there exists a $2k^2$-colouring of $\mathbb{R}^d$ containing no monochromatic homothetic copy of $rL^d$ for any $r\in R$.
\end{proposition}

\begin{proof}
    Suppose that $a_1,a_2,a_3 \in \mathbb{R}^d$ form a homothetic copy of $rL^d$ for some $r\in R$, and label these points so that $\| a_1 - a_3 \|_2=2r$. If the points $a_1,a_2,a_3$ are at a distance of $x_1,x_2,x_3$, respectively, from the origin $0$, and the angle $a_1a_20$ is $\theta$, then by the parallelogram law we have $x_1^2+x_3^2 = 2(x_2^2+r^2)$. (This equation also follows from the fact that $x_1^2= x_2^2 + r^2 -2 x_2r \cos \theta$ and $x_3^2= x_2^2 + r^2 +2 x_2r \cos \theta$.)
    
    We will give a colouring of $\mathbb{R}_{\ge 0}$ and then colour each point in $\mathbb{R}^d$ by the square of its distance to the origin. It is therefore enough to find a colouring of $\mathbb{R}_{\ge 0}$ containing no monochromatic triple of distinct points $x_1,x_2,x_3$ with $x_1+x_3 = 2(x_2+r^2)$ for some $r\in R$.
    
    First colour $\mathbb{N}$ by iterating through $2k^2$ different colours in order; that is, make each of the $2k^2$ distinct residues mod$(2k^2)$ a monochromatic colour class. Then consider the colouring of $\mathbb{R}_{\ge 0}$ where each element $x\in \mathbb{R}_{\ge 0}$ receives the same colour as the element $y\in \mathbb{N}$ which it is closest to; if there are two choices, then let $y$ be the smaller of the two. We say that $x$ is \textit{coloured by} this point $y$. Now suppose towards a contradiction that there is a monochromatic triple of distinct points $x_1,x_2,x_3$, with $x_1+x_3 = 2(x_2+r^2)$ for some $r\in R$. Suppose that $x_1, x_2, x_3$ are coloured by $y_1, y_2, y_3 \in \mathbb{N}$, respectively. By the definition of the colouring, \begin{align*}
        |\left(x_1+x_3-2x_2\right)-\left(y_1+y_3-2y_2\right)|\leq |x_1-y_1|+|x_3-y_3|+2|x_2-y_2| \leq 2.
    \end{align*}
    \noindent In fact we must have a strict inequality somewhere, since otherwise $x_i = y_i+1/2$ for each $i \in \{1,2,3\}$, and the left-hand side of the equation above evaluates to~$0$. 
    
    Since $y_1, y_2, y_3$ are integers, it follows that there exists $i \in \{-1,0,1\}$ so that $y_1+y_3-2y_2=2r^2+i$. Since $y_1, y_2, y_3$ all receive the same colour, both $y_1-y_2$ and $y_3-y_2$ are divisible by $2k^2$. Thus $i=0$ and $r^2$ is divisible by $k^2$. This contradicts the fact that $R\cap k\mathbb{Z} = \emptyset$.
\end{proof}

Let us return to colourings of the plane $\mathbb{R}^2$. Theorems~\ref{poly dist} and~\ref{prime dist} exclude finite colourings which avoid monochromatic pairs of points whose distance is, respectively, a value of a fixed polynomial with integer coefficients, or a prime.
It is natural to ask what other sets of forbidden distances ensure that there is no finite colouring of the plane avoiding these distances.
Soifer~\cite{soifer2010between} asked about when the forbidden distances are powers of $2$ and Kahle~\cite{soifer2010between} for when the forbidden distances are factorials.

\begin{problem}[Soifer {\cite{soifer2010between}}]
Is there a finite colouring of $\mathbb{R}^2$ containing no monochromatic pair of points whose distance is equal to $2^n$ for some $n\in \mathbb{N}$?
\end{problem}

\begin{problem}[Kahle {\cite{soifer2010between}}]
Is there a finite colouring of $\mathbb{R}^2$ containing no monochromatic pair of points whose distance is equal to $n!$ for some $n\in \mathbb{N}$?
\end{problem}
    
In fact, we do not know of a single infinite subset of $\mathbb{N}$ for which the plane has a finite colouring avoiding those distances. 
While we conjecture that both of the problems above should have negative answers, it appears challenging to extend current methods to solve these two problems due to the exponential growth of the forbidden distances.
Of course, it would be more exciting if either of these two problems has a positive answer.

F{\"u}rstenberg, Katznelson, and Weiss~\cite{furstenberg1990ergodic} proved that if $A\subset \mathbb{R}_{>0}$ is such that $\sup A = \infty$, then there is no finite measurable colouring of the plane avoiding a monochromatic pair of points whose distance is contained in $A$. It is not known if the measurable condition is necessary, but Bukh~\cite{bukh2008measurable} conjectures that it is, in particular when $A$ is algebraically independent.

\begin{conjecture}[Bukh {\cite{bukh2008measurable}}]
Let $A\subset \mathbb{R}_{>0}$ be algebraically independent.
Then there is a finite colouring of $\mathbb{R}^2$ containing no monochromatic pair of points whose distance is contained in $A$.
\end{conjecture}

Another natural question is whether new or improved methods could result in better bounds for the chromatic number or even the independence density of the considered Cayley graphs.
Perhaps the simplest such Cayley graph to consider is the following simplified version of the Cayley graph used to show that the odd distance graph has infinite chromatic number~\cite{davies2022odd}.
For each positive integer $k$, let $G_k=G(\mathbb{Z}^2,(2 \mathbb{Z} +1) \{ (2^{k-j}, 2^{k+j} ) : j=1, \ldots, k-1\})$. The methods used for the odd distance graph show that $\chi(G_k) \ge \left\lceil \frac{k+1}{2} \right\rceil$ and $\overline{\alpha}(G_k) \le \frac{2}{k+1}$.

\begin{problem}
    What is the asymptotic behavior of $(\chi(G_k))_{k=1}^\infty$ and $(\overline{\alpha}(G_k))_{k=1}^\infty$?
\end{problem}

A simple upper bound for the chromatic number is $\chi(G_k) \le 2^{k-1}$. This is obtained by observing that for each $j \in \{1,\ldots,k-1\}$, the subgraph $G(\mathbb{Z}^2,(2 \mathbb{Z} +1)   (2^{k-j}, 2^{k+j}) )$ of $G_k$ is bipartite. So, $G_k$ can be partitioned into $k-1$ bipartite subgraphs, resulting in the bound of $\chi(G_k) \le 2^{k-1}$ by considering a product colouring.
For the independence density, this also gives the lower bound of $\overline{\alpha}(G_k) \ge  \frac{1}{2^{k-1}}$.
It would be particularly interesting to determine whether or not $(\chi(G_k))_{k=1}^\infty$ grows polynomially.

\section*{Acknowledgements} We would like to thank the referee for many helpful comments which have greatly improved the presentation of this paper.

\bibliographystyle{plain}

\end{document}